\newcommand*\samethanks[1][\value{footnote}]{\footnotemark[#1]}
\documentclass{article}
\usepackage{amssymb}
\usepackage{amsthm}
\usepackage{amsmath}
\usepackage{xcolor}
\usepackage{tikz}
\usepackage{authblk}
\usepackage{a4wide} 
\newtheorem{Theorem}{Theorem}[section]
\newtheorem{Lemma}[Theorem]{Lemma}
\newtheorem{Definition}[Theorem]{Definition}
\newtheorem{Claim}[Theorem]{Claim}
\newtheorem{Proposition}[Theorem]{Proposition}
\newtheorem{Corollary}[Theorem]{Corollary}
\newtheorem{Observation}[Theorem]{Observation}

\theoremstyle{definition} \newtheorem{Remark}[Theorem]{Remark}

\allowdisplaybreaks
\begin{document}
\title{The Toucher-Isolator game}
\author[1]{Chris Dowden\thanks{Supported by
Austrian Science Fund (FWF): P27290.}} 
\author[1]{Mihyun Kang\thanks{Supported by
Austrian Science Fund (FWF): P27290 and W1230.}}
\author[2]{Mirjana Mikala\v{c}ki\thanks{Partly supported by Ministry of Education, Science and Technological Development, Republic of Serbia, Grant nr. 174019}}
\author[2]{Milo\v{s} Stojakovi\'{c}\samethanks}
\affil[1]
{Institute of Discrete Mathematics, 
Graz University of Technology,
Austria.
Email: dowden@math.tugraz.at,
kang@math.tugraz.at.}
\affil[2]
{Department of Mathematics and Informatics, Faculty of Sciences, University of Novi Sad, Serbia.
Email: mirjana.mikalacki@dmi.uns.ac.rs, milos.stojakovic@dmi.uns.ac.rs.}
\setlength{\unitlength}{1cm}
\maketitle

\begin{abstract}
We introduce a new positional game called `Toucher-Isolator', 
which is a quantitative version of a Maker-Breaker type game.
The playing board is the set of edges of a given graph $G$,
and the two players, Toucher and Isolator, claim edges alternately.
The aim of Toucher is to `touch' as many vertices as possible 
(i.e.~to maximise the number of vertices
that are incident to at least one of her chosen edges), 
and the aim of Isolator is to minimise the number of vertices that are so touched.

We analyse the number of untouched vertices $u(G)$ at the end of the game when both Toucher and Isolator play optimally, obtaining results both for general graphs and for particularly
interesting classes of graphs,
such as cycles, paths, trees, and $k$-regular graphs.
We also provide tight examples.
\end{abstract}

\section{Introduction}

\subsection{Background and motivation}

One of the most fundamental and enjoyable mathematical activities is to play and analyse games,
ranging from simple examples such as snakes and ladders and noughts and crosses
to much more complex games like chess and bridge.


Many of the most natural and interesting games to play involve
pure skill, perfect information, and a sequential order of play.
These are known formally as `combinatorial' games, see e.g.~\cite{winning},
and popular examples include
Connect Four, Hex, noughts and crosses, draughts, chess, and go.

Often,
a combinatorial game might consist of two players
alternately `claiming' elements of the playing board
(e.g.~noughts and crosses, but not chess) 
with the intention of forming specific winning sets,
and such games are called `positional' combinatorial games
(for a comprehensive study, see~\cite{beckbook} or~\cite{milosbook}).
In particular,
much recent research has involved positional games
in which the board is the set of edges of a graph,
and where the aim is to claim edges
in order to form subgraphs with particular properties.

A pioneering paper in this area was that of
Chv\'{a}tal and Erd\H{o}s~\cite{chv},
in which the primary target was to form a spanning tree.
Subsequent work has then also involved other standard graph structures and properties,
such as
cliques~\cite{bal, geb12},
perfect matchings~\cite{mik, hef09},
Hamilton cycles~\cite{hef09,kriv11},
planarity~\cite{hef08},
and given minimum degree~\cite{geb09}.
Part of the appeal of these games is that there are several different versions.
Sometimes, in the so-called strong games, both players aim to be the first to form a winning set
(c.f.~three-in-a-row in a game of noughts and crosses).
In others,
only Player $1$ tries to do this,
and Player $2$ simply seeks to prevent her.

This latter class of games are known as `Maker-Breaker' positional games. A notable result here is the Erd\H{o}s-Selfridge Theorem~\cite{erdself},
which establishes a simple but general condition for the existence of a winning strategy for Breaker
in a wide class of such problems.
A quantitative generalisation of this format then involves games in which
Player $1$ aims to form as many winnning sets as possible,
and Player $2$ tries to prevent this (i.e.~Player $2$ seeks to minimise the number of winning sets formed by Player $1$).

In this paper,
we introduce a new quantitative version of a Maker-Breaker style positional game,
which we call the `Toucher-Isolator' game.
Here, the playing board is the set of edges of a given graph,
the two players claim edges alternately, the aim of Player $1$ (Toucher)
is to `touch' as many vertices as possible (i.e.~to maximise the number of vertices
that are incident to at least one of her edges), and the aim of Player $2$ (Isolator)
is to minimise the number of vertices that are touched by Toucher (i.e.~to claim all edges incident to a vertex, and do so for as many vertices as possible).

This problem is thus simple to formulate and seems very natural, with connections to other interesting games,
such as claiming spanning subgraphs, matchings, etc. 
In particular,
we note that it is related to the well-studied Maker-Breaker vertex isolation game (introduced by Chv\'{a}tal and Erd\H{o}s~\cite{chv}), where Maker's goal is to claim all edges incident to a vertex,
and it is hence also related to the positive min-degree game (see~\cite{balplu, milosbook, hef11}), where Maker's goal is to claim at least one edge of every vertex.

Our Toucher-Isolator game can be thought of as a quantitative version of these games, where Toucher now wants to claim at least one edge on \emph{as many} vertices as possible, while Isolator aims to isolate \emph{as many} vertices as possible. However, the game has never previously been investigated,
and so there is a vast amount of unexplored territory here, with many exciting questions.
What are the best strategies for Toucher and Isolator?
How do the results differ depending on the type of graph chosen?
Which graphs provide the most interesting examples?

\subsection{Results}

Given a graph $G=(V(G),E(G))$,
we use $u(G)$ to denote the number of untouched vertices
at the end of the game when both Toucher and Isolator play optimally.
We obtain both upper and lower bounds on $u(G)$,
some of which are applicable to all graphs
and some of which are specific to particular classes of graphs
(e.g.~cycles or trees).
For every one of these,
we also demonstrate that the bounds are tight
by providing examples of graphs which satisfy them exactly
(in most cases, we in fact give a family of tight examples
to show that there are infinitely many values of $n=|V(G)|$
for which equality holds).
We shall now present all of these results, the proofs of which will be given later.

Clearly,
one of the key parameters in our game 
will be the degrees of the vertices
(although,
as we shall observe later,
the degree sequence alone does not fully determine the value of $u(G)$).
In our bounds for general $G$,
perhaps the most significant is the upper bound of Theorem~\ref{gen2}.
Here,
we find that it suffices just to consider the vertices with degree at most three
(we again re-iterate that all our bounds are tight).

\begin{Theorem} \label{gen2}
	For any graph $G$, we have
	\begin{displaymath}
		d_{0} + \frac{1}{2}d_{1} - 1 \leq
		u(G) \leq d_{0} + \frac{3}{4}d_{1} + \frac{1}{2}d_{2} + \frac{1}{4}d_{3},
	\end{displaymath}
where $d_{i}$ denotes the number of vertices with degree exactly $i$.
\end{Theorem}

A notable consequence of this result is that there will be no untouched vertices
for any graph with minimum degree at least four. 
We shall later see 
(in Theorem~\ref{u>0example})
that this is not always true for graphs with minimum degree three.

For certain degree sequences,
the bounds given in Theorem~\ref{gen2} can in fact both be improved by our next result.

\begin{Theorem} \label{gen1}
For any graph $G$, we have
\begin{displaymath}
\sum_{v \in V(G)} 2^{-d(v)} - \frac{|E(G)|+7}{8} \leq
u(G) \leq \sum_{v \in V(G)} 2^{-d(v)},
\end{displaymath}
where $d(v)$ denotes the degree of vertex $v$.

Equivalently, we have
\begin{displaymath}
\sum_{i \ge 0} 2^{-i} d_i - \frac{|E(G)|+7}{8} \leq
u(G) \leq \sum_{i \ge 0} 2^{-i} d_i,
\end{displaymath}
where $d_{i}$ again denotes the number of vertices with degree exactly $i$.
\end{Theorem}

Note also that $|E(G)|$ will be small if the degrees are small,
and so Theorem~\ref{gen1} then provides a fairly narrow interval for the value of $u(G)$
(observe that Theorem~\ref{gen2} already provides a narrow interval if the degrees are large).

Moving on from these general bounds,
one very natural particular graph to consider is the cycle $C_{n}$ on $n$ vertices.
It is fascinating to play the game on such a graph
and to try to determine the optimal strategies and the proportion of untouched vertices.
We again obtain tight upper and lower bounds,
both for $C_n$ and for the closely related game on $P_n$
(the path on $n$ vertices).

\begin{Theorem} \label{cyc2}
	For all $n$, we have
\begin{displaymath}
\frac{3}{16} (n-3) \leq u(C_{n}) \leq \frac{n}{4}.
\end{displaymath}
\end{Theorem}

\begin{Theorem} \label{path1}
	For all $n$, we have
	\begin{displaymath}
		\frac{3}{16} (n-2) \leq u(P_{n}) \leq \frac{n+1}{4}.
	\end{displaymath}
\end{Theorem}

We also extend the game to general $2$-regular graphs (i.e.~unions of disjoint cycles).
Our main achievement here is to obtain a \emph{tight} lower bound of $u(G) \ge \frac{n-3}{6}$,
which (by a comparison with the lower bound of Theorem~\ref{cyc2})
also demonstrates that $u(G)$ is not solely determined by the degree sequence.

\begin{Theorem} \label{2reg2}
For any $2$-regular graph $G$ with $n$ vertices, we have
\begin{displaymath}
\frac{n-3}{6} \leq u(G) \leq \frac{n}{4}.
\end{displaymath}
\end{Theorem}

An interesting and natural extension of the game on paths is obtained by considering general trees,
although this additional freedom in the structure can make the problem significantly more challenging.
Here,
we derive the following tight bounds.

\begin{Theorem} \label{tree}
For any tree $T$ with $n>2$ vertices, we have
\begin{displaymath}
\frac{n+2}{8} \leq u(T) \leq \frac{n-1}{2}.
\end{displaymath}
\end{Theorem}

As mentioned,
it follows from Theorem~\ref{gen2}
that there will be no untouched vertices in $k$-regular graphs if $k \geq 4$,
so it is intriguing to consider the $3$-regular case.
We observe that there are $3$-regular graphs for which $u(G)=0$, and one might expect that this could be true for all such graphs. However, we in fact manage to construct a class of examples for which a constant proportion of vertices remain out of Toucher's reach.

\begin{Theorem} \label{u>0example}
For all even $n \ge 4$,
there exists a $3$-regular graph $G$ with n vertices satisfying 
\begin{displaymath}
u(G) \ge \left \lfloor \frac{n}{24} \right \rfloor.
\end{displaymath}
\end{Theorem}

\subsection{Techniques and outline of the paper}

One of our key techniques
is to analyse an appropriate \emph{`Danger'} function,
building on an idea first introduced by Erd\H{o}s and Selfridge~\cite{erdself} to prove a general
criterion for Breaker's win, the celebrated Erd\H{o}s-Selfridge Criterion. The same approach was later adapted
by Beck~\cite{beck82} for Maker, resulting in the so-called Weak Win Criterion.
In the quantitative context of our Toucher-Isolator game, it is still useful to define the
Danger function in a similar manner.

\begin{Definition} \label{Dangerdef}
	We shall say that a vertex has Danger $0$ if any of its edges have been taken by Toucher,
	and Danger $2^{-k}$ if all but $k$ of its edges have been taken by Isolator 
	and the remaining $k$ edges have not yet been taken by anyone
	(note that this includes the case when $k=0$,
	and so a vertex has Danger $1$ if all of its edges have been taken by Isolator). 
\end{Definition}

Note that the Danger function can be interpreted as the probability
that a vertex will be untouched if all of its remaining edges
are assigned to Toucher and Isolator independently and uniformly at random.

An equivalent definition is also obtained if 
we update the graph $G$ throughout the game 
by removing the edges claimed by Isolator, 
keep track of the vertices $U(G) \subseteq V(G)$ untouched by Toucher,
and define the total Danger to be $\sum_{v\in U(G)} 2^{-d(v)}$.

The following observation will be key.

\begin{Observation} \label{Dangerobs1}
	The total Danger at the start of the game is
	$\sum_{v} 2^{-d(v)}$, and the total Danger at the end of the game is precisely the number of untouched vertices.
\end{Observation}

Hence, bounds for $u(G)$ can sometimes be obtained by investigating how the total Danger changes with each move.
Here, a further observation is crucial.

\begin{Observation} \label{Dangerobs2}
	Whenever Toucher takes an edge,
	the total Danger will decrease by exactly the sum of the Dangers of the two vertices incident to this edge
	(since both of these Dangers will fall to zero).
	
	Similarly, whenever Isolator takes an edge,
	the total Danger will increase by exactly the sum of the Dangers of the two vertices incident to this edge
	(since both of these Dangers will double).
\end{Observation}

Another standard method that will be used throughout is `partition of the board'.
Here, we divide the graph up into various segments,
we focus on one particular player,
and we try to optimise that player's strategy subject to the constraint
that he/she must always take an edge from the same segment that his/her opponent has just played in
(this then provides bounds for the overall optimum strategy,
where there are no such constraints).
The main advantage of this idea is that
it enables us to split the whole graph into simpler pieces
that can be analysed more easily.
However,
we must choose the division of the graph in a rather careful manner
in order to achieve substantial results.


The remainder of the paper is structured as follows:
in Section~\ref{gen},
we prove the general bounds applicable to all graphs,
as stated in Theorem~\ref{gen2} and Theorem~\ref{gen1};
in Section~\ref{cycles},
we focus on the case when the graph is a cycle,
proving Theorem~\ref{cyc2}
(the proof of Theorem~\ref{path1} on paths is very similar,
and so is left to the appendix);
in Section~\ref{2reg},
we generalise this to any $2$-regular graph,
obtaining Theorem~\ref{2reg2};
in Section~\ref{trees},
we investigate trees,
proving Theorem~\ref{tree};
in Section~\ref{3reg},
we derive results for $3$-regular graphs,
including Theorem~\ref{u>0example};
and in Section~\ref{discussion},
we discuss some interesting remaining questions.

\section{General bounds} \label{gen}

As mentioned,
in this section
we shall now derive general bounds applicable to any graph $G$, 
proving Theorem~\ref{gen2} and Theorem~\ref{gen1}.
In each case, 
we shall also observe that there are straightforward tight examples for infinitely many values of $|V(G)|$.

We shall start with the proof of the upper bound of Theorem~\ref{gen2},
followed by tight examples,
and then give the proof of the corresponding lower bound,
again followed by tight examples.
After this,
we shall then use the same pattern for the proof of Theorem~\ref{gen1}
(and also for future sections).

We begin with perhaps the most interesting proof of this section,
which uses a variant of the partition of the board strategy involving the pairing of edges.

\begin{proof}[Proof of upper bound of Theorem~\ref{gen2}]
	We will provide Toucher with a pairing strategy to touch enough vertices for the statement to hold. 
	To do this, we will define a collection of disjoint pairs of edges, and Toucher's strategy will be to wait (and play arbitrarily) until Isolator claims an edge within a pair, and then immediately respond by claiming the other edge. 
	This way, Toucher will certainly claim at least one edge in every pair.
	
	We start by adding an auxiliary vertex and connecting it to all odd degree vertices of $G$. This will create an even graph, and so each of its components has an Eulerian tour. For each of these Eulerian tours, we then arbitrarily choose one of two orientations. Once the auxiliary vertex is removed, we are thus left with an orientation of $G$. 
	
	Let $V_i$ be the set of vertices with degree $i$, and let $V_i^{(j)}$ be the set of vertices with degree $i$ and $j$ incoming edges. 
	We shall use $d_i^{(j)}$ to denote $\left| V_i^{(j)} \right|$, 
	and we note that $\left| V_i \right|=d_i$. 
	Also, observe that for even $i$ we have $V_i = V_i^{\left(\frac{i}{2}\right)}$, while for odd $i$ we have $V_i=V_i^{\left(\frac{i+1}{2}\right)} \cup V_i^{\left(\frac{i-1}{2}\right)}$.
	
	For each vertex that has at least two incoming edges, 
	we may choose two such edges arbitrarily and pair them. 
	Note that we can do this for all vertices in $V_3^{(2)} \cup \left(\cup_{i\geq 4} V_i \right)$. 
	
	Next, for all the vertices in $V_1^{(1)} \cup V_2 \cup V_3^{(1)}$ 
	(observe that these each have exactly one incoming edge), 
	let us collect all incoming edges and pair them up arbitrarily. 
	If $\left| V_1^{(1)} \cup V_2 \cup V_3^{(1)} \right|$ is odd,
	then there will be one unpaired edge here,
	which Toucher should claim with her very first move of the game
	(before Isolator has made any moves).
	
	Note that by treating only the incoming edges at every vertex, 
	we ensure that all our edge pairs are pairwise disjoint.
	
	Let us now consider the number of vertices that Toucher will touch following this pairing strategy. 
	She certainly touches all vertices in $V_3^{(2)} \cup \left(\cup_{i\geq 4} V_i \right)$ and half (rounded up) of the vertices in $V_1^{(1)} \cup V_2 \cup V_3^{(1)}$, so counting those that remain then gives
	\begin{equation} 
	u(G) \leq d_0 + d_1^{(0)} + \frac{d_1^{(1)}}{2} + \frac{d_2}{2} + \frac{d_3^{(1)}}{2}. \label{i:1}
	\end{equation}
	
	Finally, 
	note that if we were to use the same orientation of $G$, 
	but pair the \emph{outgoing} edges instead of the incoming edges,
	then exactly the same analysis gives
	\begin{equation} 
	u(G) \leq d_0 + d_1^{(1)} + \frac{d_1^{(0)}}{2} + \frac{d_2}{2} + \frac{d_3^{(0)}}{2}. \label{i:2}
	\end{equation}
	
	Summing (\ref{i:1}) and (\ref{i:2}) (and dividing by two) then completes the proof. 
\end{proof}

For this bound,
it is trivial to note the following tight examples.

\begin{Proposition}
	Any graph with minimum degree at least four will provide a tight example to the upper bound in Theorem~\ref{gen2}.
	\qed
\end{Proposition}

Before we move on to the proof of Theorem~\ref{gen1},
which uses a Danger function approach,
let us briefly give the proof of the lower bound of Theorem~\ref{gen2},
and then also provide corresponding tight examples.

\begin{proof}[Proof of lower bound of Theorem~\ref{gen2}]
	Let $X$ denote the set of edges whose endpoints both have degree $1$,
	and let $Y$ denote the set of edges with exactly one endpoint of degree $1$. Note that $d_{1} = |Y| + 2|X|$.
	
	By giving priority to the edges in $X$,
	followed by the edges in $Y$,
	Isolator will be guaranteed to take at least
	$\left \lfloor \frac{|X|+|Y|}{2} \right \rfloor$
	of these edges in total,
	including at least
	$\left \lfloor \frac{|X|}{2} \right \rfloor$
	of the edges in $X$.
	
	Hence,
	\begin{eqnarray*}
		u(G) & \geq & d_{0} + \left \lfloor \frac{|X|+|Y|}{2} \right \rfloor + \left \lfloor \frac{|X|}{2} \right \rfloor \\
		& \geq & d_{0} + \frac{|X|+|Y|}{2} - \frac{1}{2} + \frac{|X|}{2} - \frac{1}{2} \\
		& = & d_{0} + \frac{d_{1}}{2} - 1.
	\end{eqnarray*}
\end{proof}

\begin{Proposition}
	Any graph consisting of an odd number of $K_{2}$ components
	will provide a tight example to the lower bound in Theorem~\ref{gen2}.
	\qed
\end{Proposition}

We may now proceed with the proof of Theorem~\ref{gen1},
again starting with the upper bound.
As mentioned,
this bound is obtained via an analysis of the Danger function.

\begin{proof}[Proof of upper bound of Theorem~\ref{gen1}]
	Our proof is an extension of that of the acclaimed Erd\H{o}s-Selfridge Theorem~\cite{erdself},
	which establishes conditions under which a player can obtain a `winning set' of edges.
	In the context of our game,
	we can consider Isolator as having obtained such a winning set
	if he has claimed all of the edges incident to a vertex
	(thus isolating it).
	
	However,
	for our purposes,
	it is crucial that we now also find a way to keep track	of the number of these winning sets.
	Here,
	the Danger function will play a vital role,
	with the key observation being that 
	the total Danger at the end of the game is precisely the number of untouched vertices
	(see Observation~\ref{Dangerobs1}).
	
	Let us begin by recalling (from Observation~\ref{Dangerobs2})
	that whenever Toucher takes an edge,
	the Dangers of the two incident vertices will both fall to zero,
	and	that the total Danger will hence decrease by their sum.
	By contrast,	
	whenever Isolator takes an edge,
	the total Danger will increase by exactly the sum of the Dangers of the two incident vertices.
	
	Hence,
	let us consider the strategy where Toucher always chooses the edge
	which maximises the sum of the Dangers of the two vertices incident to it.
	By this maximality condition
	(and Observation~\ref{Dangerobs2}),
	it then follows that in each pair of moves
	(one from Toucher and then one from Isolator),
	the total Danger can never increase.
		
	Note furthermore that if $|E(G)|$ is odd,
	then the game will end with one final (unpaired) move by Toucher,
	which also cannot increase the total Danger.
	
	Thus, recalling from Observation~\ref{Dangerobs1}
	that the total Danger at the start of the game is
	$\sum_{v} 2^{-d(v)}$, 
	and again remembering that the total Danger at the end of the game is exactly the number of untouched vertices,
	we hence obtain the desired bound.
\end{proof}

Again,
it is simple to identify tight examples.

\begin{Proposition}
Any graph consisting of an even number of $K_{2}$ components
will provide a tight example to the upper bound in Theorem~\ref{gen1}.
\qed
\end{Proposition}

\begin{Remark}
	Note that the upper bound of Theorem~\ref{gen1} 
	will be better than the upper bound of Theorem~\ref{gen2} if
	\begin{displaymath}
	\sum_{i \ge 4} 2^{3-i} d_i < 2d_1 + 2d_2 + d_3.
	\end{displaymath}
\end{Remark}

\begin{Remark}
	In some cases,
	it is possible to combine the Danger function technique
	used to prove the upper bound of Theorem~\ref{gen1} with the pairing approach used to prove the upper bound of Theorem~\ref{gen2}. 
	In particular,
	if our graph $G$ contains an induced subgraph $H$ with $\delta(H)\geq 4$, then Toucher could employ a pairing strategy on $E(H)$ to make sure that all vertices in $V(H)$ are touched, 
	while still having full liberty when playing on the edges of $E(G) \setminus E(H)$, with the aim of maximising the number of touched vertices. 
	This would automatically improve the upper bound of Theorem~\ref{gen1}
	from $\sum_{v \in V(G)} 2^{-d(v)}$
	to $\sum_{v \in V(G) \setminus V(H)} 2^{-d(v)}$. 
	Note furthermore that such a graph $H$ must exist as soon as $|E(G)|\geq 3|V(G)|$ (see~\cite{erd64}).
	
	A similar approach to improving the upper bound is to repeatedly look for individual vertices that can be taken care of by a pair of edges. 
	In particular, if there is a vertex $v$ with neighbours $u_1$ and $u_2$ such that $d(v) < d(u_1) -1$ and $d(v) < d(u_2)$, then we could pair the edges $vu_1$ and $vu_2$ (thus taking care of touching $v$), 
	and then use the Danger function technique on the edge set 
	$E(G) \setminus \{ vu_1, vu_2 \}$.
	Thus,
	the total Danger at the start of the game
	(and hence our upper bound for $u(G)$ in Theorem~\ref{gen1})
	would decrease by $2^{-d(v)}-\left(2^{-d(u_1)}+2^{-d(u_2)}\right)$.
\end{Remark}

Our final bound of this section
is obtained by a proof which again uses an adaptation of the Danger function approach 
of Erd\H{o}s and Selfridge~\cite{erdself} and Beck~\cite{beck82}.
Our arguments here will mirror those for the upper bound,
looking at the total Danger from the point of view of Isolator instead of Toucher.

\begin{proof}[Proof of lower bound of Theorem~\ref{gen1}]	
	Note, by Observation~\ref{Dangerobs2}, that the total Danger will decrease by at most $1$ with Toucher's first move,
	since the two vertices incident to the chosen edge can only have had Danger at most $\frac{1}{2}$ each.
	
	Suppose that Isolator then chooses the edge that maximises the sum of the Dangers of the two vertices incident to it.
	Let us suppose that this sum is $r$, say, and hence that Isolator's move causes the total Danger to increase by $r$.
	
	If Toucher's response is to take an edge that is disjoint to Isolator's choice, 
	then the total Danger will decrease back by at most $r$, by the maximality condition.
	
	If Toucher's edge instead shares a common vertex with Isolator's edge, then the total Danger will still only decrease by at most $r + \frac{1}{4}$,
	since the Danger of this common vertex can only have increased by at most $\frac{1}{4}$ (from $\frac{1}{4}$ to $\frac{1}{2}$)
	as a result of Isolator's move. 
	
	Hence, in this pair of moves, one from Isolator and then one from Toucher, the total Danger can only have decreased by at most $\frac{1}{4}$ altogether. 
	
	We can consider the $|E(G)|$ moves of the game as
	Toucher's first move
	(which we have seen decreases the total Danger by at most $1$),
	followed by $\left \lfloor \frac{|E(G)|-1}{2} \right \rfloor$ subsequent pairs of moves
	(which we have seen each decrease the total Danger by at most $\frac{1}{4}$
	if Isolator always uses the given strategy),
	followed possibly (if $|E(G)|$ is even)
	by one final move from Isolator
	(which cannot decrease the total Danger).
	
	Thus, if Isolator uses the given strategy,
	then the total Danger at the end of the game 
	(and hence the number of untouched vertices,
	by Observation~\ref{Dangerobs1}) 
	will be at least 
	\begin{displaymath}
	\sum_{v \in V(G)} 2^{-d(v)} - 1 - \frac{1}{4} \left \lfloor \frac{|E(G)|-1}{2} \right \rfloor. \qedhere
	\end{displaymath}
\end{proof}

\begin{Proposition}
	Any graph consisting of $P_{3}$ components plus exactly one $P_{2}$ component
	will provide a tight example to the lower bound in Theorem~\ref{gen1}.
\end{Proposition}
\begin{proof}
	Let $x$ denote the number of $P_{3}$ components in such a graph,
	and note that we then have
	$|E(G)|=2x+1$,
	$d_{1}=2x+2$,
	$d_{2}=x$,
	and $d_{i}=0$ for $i>2$.
	
	Hence,
	\[
	\sum_{v \in V(G)} 2^{-d(v)} - \frac{|E(G)|+7}{8} =  \frac{1}{2} (2x+2) + \frac{1}{4} x - \frac{(2x+1)+7}{8} = x.
	\]
	
	Toucher can ensure that the number of untouched vertices is only $x$
	by taking the only edge in the $P_{2}$ component in her first move, and then always immediately taking the remaining edge from any $P_{3}$ component on which Isolator plays.
\end{proof}

\begin{Remark} \label{genremark}
Note that the lower bound of Theorem~\ref{gen1}
will be better than the lower bound of Theorem~\ref{gen2} if
\begin{displaymath}
|E(G)| < 1 + \sum_{i \ge 2} 2^{3-i} d_i.
\end{displaymath}

As $2|E(G)|=\sum_{i \ge 1} id_{i}$,
this will occur if $d_{2}$ is sufficiently large
(e.g.~consider a path or a cycle,
in which case the lower bound of Theorem~\ref{gen2} is ineffective).
\end{Remark}

\begin{Remark} \label{genremark2}
Although the bounds given in this section involve only the degrees of the vertices,
we shall see examples later
(in Remark~\ref{degseqremark} and Section~\ref{3reg})
which show that $u(G)$ is not determined by the degree sequence alone.
It would be interesting to know of any particular properties or parameters of the graph that can tighten the interval in which $u(G)$ must be located. 
\end{Remark}

\section{Cycles} \label{cycles}

In this section, 
we consider the specific case when our graph is a cycle.
We shall start by applying Theorem~\ref{gen1}
to immediately obtain the upper bound in Theorem~\ref{cyc2},
and then the majority of this section will be devoted to deriving the lower bound.

\begin{proof}[Proof of upper bound of Theorem~\ref{cyc2}]
	This follows from Theorem~\ref{gen1}.
\end{proof}

\begin{Proposition}
The graph $C_{4}$ provides a tight example to the upper bound in Theorem~\ref{cyc2}.
\end{Proposition}
\begin{proof}
After Toucher's first move, Isolator can take the opposite edge (and then either of the two remaining edges with his second move).
\end{proof}

We may now start to work towards the lower bound in Theorem~\ref{cyc2}.
Note that Theorem~\ref{gen1} immediately provides a lower bound of around $\frac{n}{8}$, but we shall aim to do significantly better.
The key result here is Lemma~\ref{3of16lemma}, which will enable Isolator to guarantee three untouched vertices from every sixteen edges.

\begin{Lemma} \label{3of16lemma}
Isolator can guarantee that the number of untouched internal vertices in $P_{17}$ will be at least three.
\end{Lemma}
\begin{proof}
The proof will consist of dividing up the sixteen edges of $P_{17}$ into various segments.
Consequently, it will be useful to first establish two statements concerning segments of length three and five, respectively.

\begin{Claim}\label{cl1}
If it is Isolator's move and there is a segment consisting of three consecutive free edges, then he can isolate an internal vertex from this segment.
\end{Claim}
\begin{proof}
Let the edges of this segment be denoted by $e_a, e_b, e_c$. Isolator claims the edge $e_b$. In the following move, Toucher cannot claim both $e_a$ and $e_c$, 
so one of them is free for Isolator to claim in his following move. 
Hence,
he can isolate one internal vertex.
\end{proof}

\begin{Claim}\label{cl2}
If it is Isolator's move
and there is a segment consisting of five consecutive free edges
$e_a, e_b, e_c, e_d, e_f$,
then he can guarantee that at least one of the following will occur:
\begin{itemize}
\item [(a)] after Isolator and Toucher have each had two moves,
one internal vertex from this segment will now be isolated
and neither of Toucher's moves will have taken place outside this segment;
\item [(b)] after Isolator and Toucher have each had three moves,
two internal vertices from this segment will now be isolated
and exactly one of Toucher's moves will have taken place outside this segment;
\item [(c)] after Isolator and Toucher have each had three moves,
two internal vertices from this segment will now be isolated,
exactly two of Toucher's moves will have taken place outside this segment,
and neither $e_{a}$ nor $e_{f}$ will have been claimed by Toucher;
\item [(d)] after Isolator has had four moves and Toucher has had three moves,
three internal vertices from this segment will now be isolated.
\end{itemize}
\end{Claim}
\begin{proof}
Let Isolator claim the central edge $e_{c}$ with his first move.
After this,
let Isolator then use the strategy of trying to extend this edge
into a string of consecutive edges
(working solely within this segment),
by always choosing an edge immediately adjacent to his current string
until this is no longer possible.

At this point,
it must then be the case that the `left-most' edge of Isolator's string
is either $e_{a}$ or is adjacent to an edge of Toucher,
and similarly the `right-most' edge of Isolator's string
is either $e_{f}$ or is adjacent to an edge of Toucher.
Note also that the string must certainly contain at least two edges.

If Isolator's string contains exactly two edges,
then these must either be $e_{b}$ and $e_{c}$
or $e_{c}$ and $e_{d}$,
and it must be that Toucher has claimed the edges
either side of this string with her two moves.
Hence, we have (a).

If Isolator's string contains exactly three edges,
then observe that Toucher must have claimed at least one of the other two edges in this segment.
If (at the end of Toucher's third move)
Toucher has in fact claimed both of these other two edges,
then we have (b).
If (at the end of Toucher's third move)
Toucher has only claimed one of these other two edges,
then it can only be that this edge is $e_{b}$
(and the string is $e_{c},e_{d},e_{f}$)
or $e_{d}$
(and the string is $e_{a},e_{b},e_{c}$),
so we have (c).

Finally,
if Isolator's string contains at least four edges,
then we have (d).
\end{proof}

We may now prove the lemma.
We shall denote the sixteen edges of $P_{17}$ by
$\{e_1, e_2, \dots, e_{16}\}$,
and wlog
we may suppose that Toucher claims one of the edges
in $\{e_1, e_2, \dots, e_8\}$ with her first move.
We differentiate between the following cases.

\paragraph{Case 1:} Toucher claimed one of the edges $\{e_1, e_2, e_3\}$. 

Isolator splits the free edges $\{e_4,e_5,\dots, e_{16}\}$ into three sequences of consecutive edges  $S_1=\{e_4,e_5,\dots,e_8\}$, $S_2=\{e_9,e_{10},\dots, e_{13}\}$, and $S_3=\{e_{14},e_{15},e_{16}\}$. Isolator plays first in $S_1$.

If Claim~\ref{cl2} (a) is true for $S_{1}$,
then Isolator isolates one internal vertex in $S_{1}$,
and the edges in $S_2$ and $S_3$ are all still free.
So Isolator then plays in $S_2$.
By Claim~\ref{cl2}, either he can isolate two more internal vertices  there and is done, or he isolates one internal vertex in $S_2$ and  all edges in $S_3$ are still free, 
so by Claim~\ref{cl1} he can also isolate one internal vertex in $S_3$.

If Claim~\ref{cl2} (b) or (c) are true for $S_1$,
then Isolator can isolate two internal vertices in $S_{1}$,
and at most two of the edges in
$\{e_8\}\cup S_2 \cup S_3$
can have been claimed by Toucher.
Hence,
there must exist a segment of three consecutive edges
not claimed by Toucher
among the nine edges in
$\{e_8\}\cup S_2 \cup S_3 = \{e_{8},e_{9},\ldots,e_{16}\}$,
in which case Isolator can isolate another internal vertex there
(by applying Claim~\ref{cl1}).

If Claim~\ref{cl2} (d) is true for $S_{1}$,
then Isolator can isolate three internal vertices in $S_{1}$. 

\paragraph{Case 2:} Toucher claimed one of the edges in $\{e_4, e_5\}$.

Isolator splits the free edges $\{e_1,e_2, e_3, e_6, e_7, \dots, e_{16}\}$ into three sequences of consecutive edges
$S_1=\{e_6,e_7, \dots, e_{10}\}$, $S_2=\{e_{11},e_{12}, \dots, e_{16}\}$,
and $S_3=\{e_1,e_2,e_3\}$
(note that this time $|S_{2}|=6$,
but $S_{2}$ and $S_{3}$ are no longer adjacent).
Isolator again plays first in $S_1$.

If Claim~\ref{cl2} (a) or (d) are true for $S_{1}$,
then the proof is exactly the same as with Case~$1$.

If Claim~\ref{cl2} (b) or (c) hold for $S_{1}$,
then Isolator can isolate two internal vertices in $S_{1}$,
and at most two of the edges in $S_{2} \cup S_{3}$
can have been claimed by Toucher.
Since $|S_{2}|=6$ and $|S_{3}|=3$,
there must then exist a segment of three consecutive free edges in either $S_{2}$ or $S_{3}$,
in which case
we can apply Claim~\ref{cl1} and Isolator is done. 

\paragraph{Case 3:} Toucher claimed the edge $e_6$.

Isolator splits the free edges into three sequences of consecutive edges $S_1=\{e_1, e_2, \ldots, e_5\}$,
$S_2=\{e_{11}, e_{12}, \ldots , e_{16}\}$,
and $S_3=\{e_7, e_8, e_9, e_{10}\}$.
Note that we again have
$|S_{1}|=5$, $|S_{2}|=6$, and $|S_{3}| \geq 3$,
so we may apply exactly the same proof as with Case~$2$. 

\paragraph{Case 4:} Toucher claimed one of the edges in $\{e_7,e_8\}$.

Isolator splits the free edges into three sequences of consecutive edges
$S_1=\{e_9, e_{10}, \dots, e_{13}\}$
$S_2=\{e_1, e_2, \dots, e_6\}$,
and $S_3=\{e_{14}, e_{15}, e_{16}\}$.
Again, we have
$|S_{1}|=5$, $|S_{2}|=6$, and $|S_{3}| \geq 3$,
so we may again apply exactly the same proof as with Cases~$2$ and~$3$.
\end{proof}

We may now prove the lower bound from Theorem~\ref{cyc2}.
Since Lemma~\ref{3of16lemma} already guarantees three untouched vertices for every sixteen edges,
the main substance of the proof
is to deal satisfactorily with the `leftover' edges
when $n$ is not exactly divisible by $16$.

\begin{proof}[Proof of lower bound of Theorem~\ref{cyc2}]
After Toucher has made her first move,
let Isolator then partition the $n$ edges
into segments of $16$ consecutive edges,
together with one `leftover' segment of $1$ to $16$ consecutive edges,
so that the edge claimed by Toucher is the last edge in the leftover segment.

Let Isolator then use the strategy of always responding
in the same segment
in which Toucher played her previous move.
By Lemma~\ref{3of16lemma},
Isolator can thus guarantee that the number of untouched internal vertices
in each $16$-edge segment
will be at least three.

Hence,
if we use $k$ to denote the number of edges in the leftover segment,
it suffices to show that Isolator
can also guarantee isolating at least
$\frac{3}{16}(k-3)$ internal vertices here.

If $k \leq 3$,
then there is nothing to prove.

If $k \in \{4,5,6,7,8\}$,
then we need to show that Isolator can isolate at least one internal vertex.
Since Toucher's edge is the last one in this segment,
there exist at least three consecutive free edges,
so we may simply apply Claim~\ref{cl1}.

If $k \in \{9,10,11,12,13\}$,
then we need to show that Isolator can isolate at least two internal vertices.
Since Toucher's edge is the last one in this segment,
there exist at least eight consecutive free edges,
so we may split these into two sequences of consecutive edges $S_{1}$ and $S_{2}$
with $|S_{1}|=5$ and $|S_{2}|=3$.
Isolator then plays first in $S_{1}$.
By Claim~\ref{cl2},
either he can isolate two internal vertices in $S_{1}$ and is done,
or he isolates one internal vertex in $S_{1}$
and all edges in $S_{2}$ are still free,
so by Claim~\ref{cl1}
he can then also isolate one internal vertex in $S_{2}$.

If $k \in \{14,15,16\}$,
then we need to show that Isolator can isolate at least three internal vertices.
To achieve this,
we may split the thirteen consecutive free edges into three adjacent sequences of consecutive edges
$S_{1}$, $S_{2}$, and $S_{3}$
with $|S_{1}|=|S_{2}|=5$ and $|S_{3}|=3$,
and argue exactly as in Case~$1$ of Lemma~\ref{3of16lemma}.
\end{proof}

\begin{Proposition}
The graph $C_{3}$ provides a tight example to the lower bound in Theorem~\ref{cyc2}.
\qed
\end{Proposition}

In the appendix,
we use similar arguments to this section to prove Theorem~\ref{path1} on paths.

\section{$\mathbf{2}$-regular graphs} \label{2reg}

In this section, we now generalise our playing board from a cycle to a collection of disjoint cycles,
i.e.~any $2$-regular graph.
Again,
we shall start by applying Theorem~\ref{gen1}
to immediately obtain the upper bound in Theorem~\ref{2reg2},
and then we will work towards deriving the lower bound.

\begin{proof}[Proof of upper bound of Theorem~\ref{2reg2}]
	This follows from Theorem~\ref{gen1}.
\end{proof}

\begin{Proposition} \label{C4cpts}
Any graph consisting of $C_{4}$ components will provide a tight example to the upper bound in Theorem~\ref{2reg2}.
\end{Proposition}
\begin{proof}
Note that Isolator can certainly isolate one vertex from each such component
by always immediately taking the opposite edge in any $C_{4}$ on which Toucher plays
and then taking the fourth edge
as soon as Toucher takes the third edge.
\end{proof}

The proof of the lower bound in Theorem~\ref{2reg2} will involve treating the components differently depending on their size modulo $6$,
so we shall find it useful to first prove three lemmas related to this. We begin by applying Theorem~\ref{gen2} to obtain a result specific to the case when a cycle has length $k \in \{4,5,6\} \bmod 6$.

\begin{Lemma} \label{2reglem1}
Let $k \in \{ 4,5,6 \} \bmod 6$.
Then
\begin{displaymath}
u(C_{k}) \geq \frac{k}{6}.
\end{displaymath}
\end{Lemma}
\begin{proof}
Let us write $k$ as $6r+s$,
where $s \in \{ 4,5,6 \}$.
Then, by Theorem~\ref{cyc2},
we have
\[
u(C_{k}) \geq \frac{3}{16} (k-3) = \frac{3}{16} (6r+s-3) \geq \frac{3}{16} (6r+1) > r.
\]

Since $u(C_{k})$ must be an integer, we then in fact have $u(C_{k}) \geq r+1 \geq \frac{k}{6}$, and we are done.
\end{proof}

\begin{Remark}
It is also relatively simple to give a self-contained proof of this result,
rather than using Theorem~\ref{cyc2}.
\end{Remark}

Note that the bound of Lemma~\ref{2reglem1}
is certainly not valid for all $k$
(e.g.~consider $C_{3}$).
Hence,
in the next two lemmas
we shall deal separately with components of length $k \in \{1,2,3\} \bmod 6$.
We shall find it extremely helpful to consider the case when
Isolator allows Toucher to have the first two moves in such a component.

\begin{Lemma} \label{2reglem2}
Let $k \in \{ 1,2,3 \} \bmod 6$.
Then Isolator can guarantee that the number of untouched vertices in $C_{k}$ will be at least $\frac{k-3}{6}$
even if Toucher has the first two moves
(and Isolator and Toucher play alternately after this).
\end{Lemma}
\begin{proof}
Wlog (since we have a cycle), Toucher makes her first move in edge $1$.

For every $6$-edge section after this
(i.e.~edges $2$--$7$,
edges $8$--$13$,
etc.),
we can consider the six edges as two $3$-edge segments
(e.g.~edges $2$--$7$ will be considered as two $3$-edge segments $2$--$4$ and $5$--$7$).

Whenever Toucher plays in one of these $3$-edge segments,
Isolator can then immediately take the central edge of the other $3$-edge segment,
and Isolator can also always eventually take one of the edges either side of this central edge
(since when Toucher takes one,
Isolator can just immediately take the other).

Hence, Isolator can certainly always obtain two consecutive edges in each of these $6$-edge sections of the cycle, so there will be an untouched vertex each time.

Now observe that there are exactly $\left \lceil \frac{k-3}{6} \right \rceil$ such sections,
since $k \in \{ 1,2,3 \} \bmod 6$, so we are done.
\end{proof}

We shall also find it helpful to consider the case when
Isolator makes the first move in a component.

\begin{Lemma} \label{2reglem3}
Let $k \in \{ 1,2,3 \} \bmod 6$.
Then Isolator can guarantee that the number of untouched vertices in $C_{k}$ will be at least $\frac{k+3}{6}$
if Isolator plays first
(and Toucher and Isolator play alternately after this).
\end{Lemma}
\begin{proof}
The $k=3$ case can easily be verified, so let us assume that $k \geq 7$.

Let Isolator initially use the strategy of trying to extend his first edge into a long string of consecutive edges,
by always choosing an edge immediately adjacent to his current string until Toucher has `blocked' both sides of this string with edges of her own
(note that these two edges of Toucher will be distinct,
since $k>3$).

If Isolator is able to use this strategy for the entire game,
then he will finish with
$\left \lceil \frac{k}{2} \right \rceil$ consecutive edges,
and hence the number of untouched vertices will be
$\left \lceil \frac{k}{2} \right \rceil - 1$,
which is certainly greater than $\frac{k+3}{6}$
(since we are assuming that $k \geq 7$),
so we are done.
Thus, we may assume that this does not happen.

Let us therefore consider the state of the game at the time when Isolator is about to make his first move for which he is no longer able to use this strategy
(due to both sides having been blocked by Toucher).

Suppose Isolator had managed to achieve a string of $j$ consecutive edges (note it must be that $j \geq 2$),
and wlog let these be edges $2$ to $(j+1)$.
Hence, Toucher has edges $1$ and $j+2$,
and Toucher also has another $j-2$ `rogue' edges elsewhere.

Let us split the edges from $j+3$ to $k$ into $3$-edge segments
(i.e.~edges $j+3$ to $j+5$,
edges $j+6$ to $j+8$,
and so on,
ignoring the final one or two edges if $k-(j+2)$ is not congruent to $0$ mod $3$).

There will be at least $\frac{k-2-(j+2)}{3} = \frac{k-j-4}{3}$ such segments,
at most $j-2$ of which will contain one of Toucher's rogue edges.
Hence, at least
$\frac{k-j-4}{3} - (j-2) = \frac{k-4j+2}{3}$ of these segments will be `unspoilt',
in the sense that none of their edges have yet been taken by either player.

Recall that Isolator has the next move.
Hence, he can immediately take the central edge from one of the unspoilt segments,
and
(as in the proof of Lemma~\ref{2reglem2})
can also always eventually take one of the edges either side of this central edge,
thus isolating a vertex.

Whenever Toucher plays first in one of the unspoilt segments,
Isolator can then immediately take the central edge of any remaining unspoilt segment,
again eventually isolating a vertex.

Hence, we see that Isolator will be able to guarantee
at least one untouched vertex
from at least half of the
segments that were unspoilt.
Thus, he will obtain at least
$\frac{k-4j+2}{6}$ such vertices,
together with the $j-1$ vertices that he already had
from his string of $j$ consecutive edges.

Hence, the total number of untouched vertices adds up to at least
$\frac{k+2j-4}{6}$,
which is at least $\frac{k}{6}$ by our observation that $j \geq 2$,
and at least $\frac{k+3}{6}$ due to integrality
and the fact that $k \in \{1,2,3\} \bmod 6$.
\end{proof}

We are now ready to use our three lemmas to complete the proof of Theorem~\ref{2reg2}.

\begin{proof}[Proof of lower bound of Theorem~\ref{2reg2}]
Recall from Lemma~\ref{2reglem1}
that Isolator can guarantee
that at least $\frac{1}{6}$ of the vertices
from each component of size $k$ for $k \in \{4,5,6\} \bmod 6$ will be untouched.
Hence, it only remains to deal with the other components.

Let us pair up these other components into partners,
with at most one such component left over
(it will not matter whether the partners have the same size $\bmod$~$6$,
only that the sizes belong to $\{1,2,3\} \bmod 6$).

When Toucher first plays in one of a pair,
let Isolator make one move in the partner.
After this,
whenever Toucher plays again anywhere in this pair,
let Isolator respond in the same component as Toucher
(so Toucher will have the first two moves in one of the pair,
with alternate moves after this,
and Isolator will have the first move in the partner,
with alternate moves after this).

By Lemmas~\ref{2reglem2} and~\ref{2reglem3},
if two paired components have size $k_{1}$ and $k_{2}$, respectively,
then Isolator can guarantee that the number of untouched vertices in these two components
will be at least
$\frac{k_{1}-3}{6} + \frac{k_{2}+3}{6} = \frac{k_{1}+k_{2}}{6}$.
Thus,
Isolator can guarantee that at least $\frac{1}{6}$ of the vertices
from each pair will be untouched.

By then applying Lemma~\ref{2reglem2} as a lower bound
for the number of untouched vertices
in the leftover component (if one exists),
we hence obtain our result.
\end{proof}

\begin{Proposition} \label{oddC3cpts}
Any graph consisting of an odd number of $C_{3}$ components
will provide a tight example to the lower bound in Theorem~\ref{2reg2}.
\end{Proposition}
\begin{proof}
Note that Toucher can make the first move in Component~$1$, say,
and can then pair up the remaining components
to ensure that she can also make the first move in half of these.

In every component
in which Toucher made the first move,
she can guarantee eventually taking a second edge
and hence leaving no untouched vertices.
In every other component,
she can guarantee eventually taking one edge
and hence leaving only one untouched vertex.
\end{proof}

\begin{Remark} \label{degseqremark}
Note that Theorem~\ref{cyc2} implies that
the lower bound of Theorem~\ref{2reg2} will not be tight for $C_n$ if $n>3$.
Thus,
as mentioned earlier in Remark~\ref{genremark2},
this observation together with the tight example of Proposition~\ref{oddC3cpts}
shows that $u(G)$ is not solely determined by the degree sequence
(see also Section~\ref{3reg}).
\end{Remark}

\section{Trees} \label{trees}

In the previous section, we explored one way of generalising the playing board
from the cycles and paths considered in Theorem~\ref{cyc2} and Theorem~\ref{path1},
by investigating general $2$-regular graphs.
In this section,
we consider another natural extension, by instead examining general trees.
We start by proving the upper bound of Theorem~\ref{tree}
and providing a family of tight examples,
and then we also prove the lower bound and give a tight example.

\begin{proof}[Proof of upper bound of Theorem~\ref{tree}]	
By Theorem~\ref{gen1}, we have
\begin{displaymath}
u(T) \leq \sum_{v \in V(T)} 2^{-d(v)}.
\end{displaymath}
Note that (since $T$ can have no vertices of degree $0$)
the sum on the right-hand-side is maximised
when all but one of the vertices have degree $1$,
since otherwise one can always achieve a higher value
by decreasing the second largest degree by $1$ and increasing the largest degree by $1$.

Hence, we obtain
\begin{displaymath}
u(T) \leq \frac{n-1}{2} + 2^{1-n}.
\end{displaymath}

But since $n \geq 3$,
we have $2^{1-n} < \frac{1}{2}$,
so the integrality of $u(T)$
then implies that we must actually have
$u(T) \leq \frac{n-1}{2}$.
\end{proof}

\begin{Proposition} \label{treeex1}
Any star with an odd number of vertices
will provide a tight example to the upper bound in Theorem~\ref{tree}.
\qed
\end{Proposition}

We now move on to the lower bound.

\begin{proof}[Proof of lower bound of Theorem~\ref{tree}]
In the main part of the proof, we shall work towards showing
\begin{equation} \label{tree2eqnb}
u(T) \geq \frac{n+d_{1}-1}{8}.
\end{equation}
The result will then follow from a combination of~\eqref{tree2eqnb}, Theorem~\ref{path1},
and one special case that will need to be considered separately.
	
In order to establish~\eqref{tree2eqnb}, we shall proceed by
first (i)~analysing the proof of the lower bound of Theorem~\ref{gen1}
to see that some aspects can be improved slightly
when the graph is known to be a tree,
then (ii)~obtaining a useful result on the average degree of the non-leaves,
and finally (iii)~using this to optimise our bound. 

\paragraph{(i)}
Recall that the proof of the lower bound of Theorem~\ref{gen1}
utilised the concept of the Danger of a vertex.
The bound obtained then followed from showing that
the total Danger will decrease by at most $1$
with Toucher's first move,
and then by at most $\frac{1}{4}$
with every subsequent pair of moves
if Isolator uses the tactic of always choosing 
the edge which maximises the sum of the Dangers of the two vertices incident to it. 

However,
it can immediately be seen that for a tree with $n>2$ vertices,
the total Danger can actually only decrease by at most $\frac{3}{4}$ with Toucher's first move,
since there cannot be two adjacent leaves.
Hence,
we can certainly add an extra $1- \frac{3}{4} = \frac{1}{4}$
to the lower bound obtained in Theorem~\ref{gen1}.
We shall now also show that the total Danger can only decrease by at most $\frac{1}{8}$
with the first subsequent pair of moves,
meaning that we can then add a further
$\frac{1}{4} - \frac{1}{8} = \frac{1}{8}$
to this bound.

To see this,
first note that
(with the stated tactic)
Isolator will certainly take an unplayed edge $uw$
incident to a leaf $w$ on his first move.
If we use $\textrm{D}(z)$
to denote the Danger of a vertex $z$ after Toucher's first move,
then Isolator's move thus causes the total Danger to temporarily increase by $\frac{1}{2} + \textrm{D}(u)$.

In order for the total Danger to then decrease back by more than
$\frac{5}{8} + \textrm{D}(u)$
with Toucher's next move,
note that she would have to take an adjacent edge $uv$
(due to the maximality condition in Isolator's strategy)
satisfying
$2 \textrm{D}(u) + \textrm{D}(v) > \frac{5}{8} + \textrm{D}(u)$,
i.e.~$\textrm{D}(u) + \textrm{D}(v) > \frac{5}{8}$.
	
Since $u$ cannot be a leaf
(as it is adjacent to the leaf $w$),
this is only possible if
$\textrm{D}(u) = \frac{1}{4}$
and $\textrm{D}(v) = \frac{1}{2}$.
But in this case,
the entire tree $T$
would consist of just the path $wuv$,
which would contradict the fact that
Toucher has already been able to take one edge somewhere with her first move!

Hence,
we find that we are indeed able to add the promised increments
to the lower bound given in Theorem~\ref{gen1}
if $T$ is a tree (with $|V(T)|>2$),
and we thus obtain
\begin{equation}
u(T) \geq \sum_{v \in V(T)} 2^{-d(v)} - \frac{|E(T)|+7}{8} + \frac{1}{4} + \frac{1}{8}
 =  \frac{d_{1}}{2} + \sum_{v:d(v) \geq 2} 2^{-d(v)} - \frac{n}{8} - \frac{3}{8}. \label{tree2eqnc}
\end{equation}

\paragraph{(ii)} We shall now work towards our aforementioned result
on the average degree of the non-leaves of $T$.
Let us first recall that
(since $n>2$)
any two leaves must be non-adjacent,
and so it is then clear that $u(T) \geq \frac{d_{1}-1}{2}$.
Hence,~\eqref{tree2eqnb}
is certainly satisfied
if $d_{1} \geq \frac{n}{3} + 1$,
so we may assume that $d_{1} < \frac{n}{3} + 1$.

Now let $x$ denote the average degree of the $n-d_{1}$ non-leaves,
and observe that
$d_{1} + x(n-d_{1}) = 2n-2$,
and
$
x = 1 + \frac{n-2}{n-d_{1}}.
$
Thus,
since $d_{1} < \frac{n}{3} + 1 < \frac{n}{2} + 1$,
we have $x<3$. 
	
\paragraph{(iii)} We shall now utilise our bound on $x$ in conjunction with~\eqref{tree2eqnc}.

Note that (for given $d_{1}$)
the sum
$\sum_{v:d(v) \geq 2} 2^{-d(v)}$
is minimised
when the non-leaves all have degrees differing by at most $1$,
since otherwise one can always achieve a lower value
by increasing the smallest non-leaf degree by $1$
and decreasing the largest non-leaf degree by $1$.
	
Hence,
since $x<3$,
we find that the sum
$\sum_{v:d(v) \geq 2} 2^{-d(v)}$
is minimised
(for given $d_{1}$)
when the non-leaves all have degree $2$ or $3$.
In this case,
we have
$d_{1} + 2d_{2} + 3d_{3} = 2(n-1)$
and $d_{2} = n-d_{1}-d_{3}$,
and so we obtain
$d_{2} = n+2-2d_{1}$
and $d_{3}=d_{1}-2$.

Thus,~\eqref{tree2eqnc} then gives
\[
	u(T)  \geq \frac{d_{1}}{2} + \frac{n+2-2d_{1}}{4} + \frac{d_{1}-2}{8} - \frac{n}{8} - \frac{3}{8} = \frac{n+d_{1}-1}{8},
\]
as desired. \\

If $d_{1} > 2$,
then we are done.
If not,
then $T$ must be a path,
and we can look to apply our lower bound
$u(P_{n}) \geq \frac{3}{16} (n-2)$
from Theorem~\ref{path1}.

We certainly have
$\frac{3}{16} (n-2) \geq \frac{n+2}{8}$ for $n \geq 10$,
and it can also be checked that
$\left \lceil \frac{3}{16} (n-2) \right \rceil \geq \frac{n+2}{8}$
for $n \in \{3,4,5,6,8,9\}$,
leaving only the case when $T=P_{7}$.

For this final case,
it suffices to show that
Isolator can always guarantee that
at least two of the vertices will be untouched,
and it can be checked that this is indeed so.
\end{proof}

\begin{Proposition} \label{treeex2}
The graph $P_{6}$
provides a tight example to the lower bound in Theorem~\ref{tree}.
\end{Proposition}
\begin{proof}
This follows immediately from Theorem~\ref{path1}.
\end{proof}

\section{$\mathbf{3}$-regular graphs} \label{3reg}

Recall that in Section~\ref{2reg}
we considered the case when our playing board is a $2$-regular graph.
The natural generalisation of this is to consider $k$-regular graphs for $k>2$.
However,
we already know from Theorem~\ref{gen2}
that $u(G)=0$ for all $k$-regular $G$ when $k>3$.
Hence, it only remains to now deal with the case when $k=3$. 

We start by giving an upper bound for $u(G)$,
then we focus on constructing $3$-regular examples with $u(G)>0$
(proving Theorem~\ref{u>0example}),
and then finally we observe that there are also $3$-regular examples with $u(G) = 0$.

As mentioned,
we begin with our best known upper bound, 
which is a direct consequence of Theorem~\ref{gen1}.

\begin{Corollary} \label{3regcor}
For any $3$-regular graph $G$ with $n$ vertices, we have
\begin{displaymath}
u(G) \leq \frac{n}{8}.
\end{displaymath}
\end{Corollary}
\begin{proof}
	This follows immediately from Theorem~\ref{gen1}.
\end{proof}

It is not at all straightforward to construct any $3$-regular graphs with $u(G)>0$. However, the following example shows that there do indeed exist such graphs.

\begin{Proposition} \label{3regex}
The $3$-regular graph $G$ depicted in Figure~\ref{3regfig1} will have an untouched vertex. 
\begin{figure} [ht]
\setlength{\unitlength}{1cm}
\centering
\begin{tikzpicture}
\draw(2/3,-2.75) node{\Large{$G$}};

\draw (0,0.5) -- (1/3,3.5/3) -- (1,3.5/3) node[midway, above]{$\large{H_{1}}$} -- (4/3,0.5) -- (2/3,-0.5/3) -- (0,0.5);
\filldraw [black]
(2/3,-0.5/3) circle (2pt)
(0,0.5) circle (2pt)
(4/3,0.5) circle (2pt)
(1/3,3.5/3) circle (2pt)
(1,3.5/3) circle (2pt)
(1.5/3,2/3) circle (2pt)
(2/3,1/3) circle (2pt)
(2.5/3,2/3) circle (2pt);

\draw (2/3,1/3) -- (1.5/3,2/3) -- (2.5/3,2/3) -- (2/3,1/3);
\draw (2/3,1/3) -- (2/3,-0.5/3);
\draw (1.5/3,2/3) -- (1/3,3.5/3);
\draw (2.5/3,2/3) -- (1,3.5/3);

\draw (1.5,-1.5) -- (5.5/3,-2.5/3) -- (2.5,-2.5/3) -- (8.5/3,-1.5) -- (6.5/3,-6.5/3) node[midway, below right]{$\large{H_{3}}$} -- (1.5,-1.5);
\filldraw [black]
(6.5/3,-6.5/3) circle (2pt)
(1.5,-1.5) circle (2pt)
(8.5/3,-1.5) circle (2pt)
(5.5/3,-2.5/3) circle (2pt)
(2.5,-2.5/3) circle (2pt)
(6/3,-4/3) circle (2pt)
(6.5/3,-5/3) circle (2pt)
(7/3,-4/3) circle (2pt);

\draw (6.5/3,-5/3) -- (6/3,-4/3) -- (7/3,-4/3) -- (6.5/3,-5/3);
\draw (6.5/3,-5/3) -- (6.5/3,-6.5/3);
\draw (6/3,-4/3) -- (5.5/3,-2.5/3);
\draw (7/3,-4/3) -- (2.5,-2.5/3);

\draw (-1.5,-1.5) -- (-3.5/3,-2.5/3) -- (-0.5,-2.5/3) -- (-0.5/3,-1.5) -- (-2.5/3,-6.5/3) -- (-1.5,-1.5) node[midway, below left]{$\large{H_{2}}$} ;
\filldraw [black]
(-2.5/3,-6.5/3) circle (2pt)
(-1.5,-1.5) circle (2pt)
(-0.5/3,-1.5) circle (2pt)
(-3.5/3,-2.5/3) circle (2pt)
(-0.5,-2.5/3) circle (2pt)
(-3/3,-4/3) circle (2pt)
(-2.5/3,-5/3) circle (2pt)
(-2/3,-4/3) circle (2pt);

\draw (-2.5/3,-5/3) -- (-3/3,-4/3) -- (-2/3,-4/3) -- (-2.5/3,-5/3);
\draw (-2.5/3,-5/3) -- (-2.5/3,-6.5/3);
\draw (-3/3,-4/3) -- (-3.5/3,-2.5/3);
\draw (-2/3,-4/3) -- (-0.5,-2.5/3);

\draw (-0.5/3,-1.5) -- (1.5,-1.5)  node[midway, above]{$e_{23}$};
\draw (8.5/3,-1.5) .. controls (3.5,-1) and (2,0.5) .. (4/3,0.5) node[pos=0.6, right]{$e_{13}$};
\draw (-1.5,-1.5) .. controls (-6.5/3,-1) and (-2/3,0.5) .. (0,0.5) node[pos=0.6, left]{$e_{12}$};

%
%
\end{tikzpicture}
\caption{A $3$-regular graph $G$ satisfying $u(G) \geq 1$.} \label{3regfig1}
\end{figure}
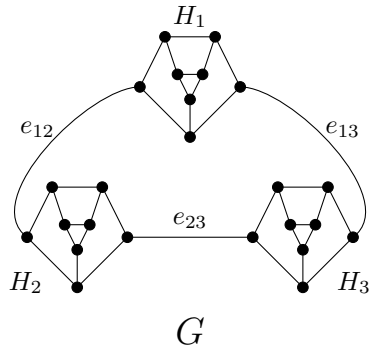
\end{Proposition}
\begin{proof}
First,
let us observe that $G$ consists of three identical blocks
$H_{1}$, $H_{2}$, and $H_{3}$,
together with the edges $e_{12}$, $e_{13}$, and $e_{23}$.
Thus, by symmetry, wlog Toucher makes her first move somewhere in
$H_{1} \cup e_{12}
\cup e_{13}
$.
Let Isolator then take the edge $e_{23}$,
and note that wlog Toucher's next move is not in $H_{3}$.

From this point on,
we shall just focus on the graph
$H_{3}$,
as shown in Figure~\ref{3regfig2}.
Recall that Isolator has already taken the edge $e_{23}$,
all the edges in $H_{3}$ are as yet unplayed
(it will not matter whether or not the edge $e_{13}$ has been taken),
and Isolator has the next move.
Let Isolator use this move to take the internal edge $u_1u_2$ marked with an $I$.

\begin{figure} [ht]
\setlength{\unitlength}{1cm}
\centering
\begin{tikzpicture}
\filldraw [black]
(2,-0.5) circle (2pt) node[anchor=north] {$v_{3}$}
(0,1.5) circle (2pt) node[anchor=north] {$v_{4}$}
(2,1) circle (2pt) node[anchor=east] {$u_{1}$}
(1.5,2) circle (2pt) node[anchor=east] {$u_{2}$}
(2.5,2) circle (2pt) node[anchor=west] {$u_{3}$}
(4,1.5) circle (2pt)
(1,3.5) circle (2pt) node[anchor=south] {$v_{1}$}
(3,3.5) circle (2pt) node[anchor=south] {$v_{2}$};
\draw (0,1.5) -- (1,3.5) node[midway, left]{$3b$} -- (3,3.5) node[midway, above]{$3d$} -- (4,1.5) node[midway, right]{$3e$} -- (2,-0.5) node[midway, below right]{$3c$} -- (0,1.5) node[midway, below left]{$3a$};
\draw (2,1) -- (1.5,2) node[midway, left]{$I$} -- (2.5,2) node[midway, above]{$2a$} -- (2,1) node[midway, right]{$1a$};
\draw (0,1.5) -- (-1,1.5) node[midway, above]{$e_{23}$} node[midway, below]{$I$};
\draw (4,1.5) -- (5,1.5) node[midway, above]{$e_{13}$} node[midway, below]{$?$};
\draw (2,-0.5) -- (2,1) node[pos=0.65, right]{$1b$};
\draw (1.5,2) -- (1,3.5) node[pos=0.6, right]{$2b$};
\draw (2.5,2) -- (3,3.5) node[pos=0.6, left]{$2c$};
\end{tikzpicture}
\caption{The graph $H_{3}$.} \label{3regfig2}
\end{figure}
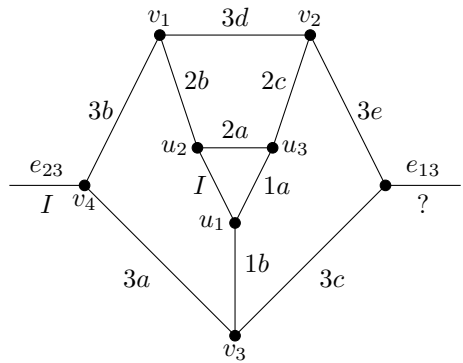

\paragraph{Case (i):} Toucher does not take one of the `inner'-edges
(i.e.~those labelled with a $1$ or a $2$) in her next move. 

Then all these inner edges are still unplayed, and the inner-vertices $u_{1}$, $u_{2}$, and $u_{3}$ are still untouched.
Thus, Isolator may then take $1a$, Toucher is forced to take $1b$ (to avoid $u_{1}$ becoming isolated),
Isolator may then take $2a$, Toucher is forced to take $2b$ (to avoid $u_{2}$ becoming isolated),
and Isolator may then take $2c$ and hence isolate $u_{3}$.

\paragraph{Case (ii):} Toucher takes one of the edges labelled with a $1$ in her next move. 

Then all the edges labelled with a $2$ or a $3$ are still unplayed,
and the vertices $u_{2}$, $v_{1}$, $v_{2}$, and $v_{4}$ are still untouched.
Thus, Isolator may then take $2b$,
Toucher is forced to take $2a$ (to avoid $u_{2}$ becoming isolated), Isolator may then take $3b$, Toucher is forced to take $3d$
(to avoid $v_{1}$ becoming isolated), and Isolator may then take $3a$ and hence isolate $v_{4}$.

\paragraph{Case (iii):} Toucher takes one of the edges labelled with a $2$ in her next move. 

Then all the edges labelled with a $1$ or a $3$ are still unplayed, and the vertices $u_{1}$, $v_{3}$, and $v_{4}$ are still untouched.
Thus, Isolator may then take $1b$, Toucher is forced to take $1a$ (to avoid $u_{1}$ becoming isolated),
Isolator may then take $3a$, Toucher is forced to take $3c$ (to avoid $u_{3}$ becoming isolated),
and Isolator may then take $3b$ and hence isolate $v_{4}$.
\end{proof}

Using Proposition~\ref{3regex},
we may now prove Theorem~\ref{u>0example}.

\begin{proof}[Proof of Theorem~\ref{u>0example}]
Note that the graph $G$ in Proposition~\ref{3regex} has $24$ vertices.
Hence, we may simply take $\left \lfloor \frac{n}{24} \right \rfloor$ components identical to $G$,
and any $3$-regular graph on the other vertices.
\end{proof}

Recall that in the $2$-regular case
(see Theorem~\ref{2reg2}), 
the only graph for which $u(G)=0$ is the triangle.
However, it turns out that there are infinitely many $3$-regular graphs for which there will be no untouched vertices.

\begin{Proposition}
Any graph consisting of $K_{4}$ components will have no untouched vertices.
\end{Proposition}
\begin{proof}
Whenever Isolator plays first in a component (taking the edge $v_{1}v_{2}$, say), Toucher can then immediately take the non-adjacent edge from this same component
(let us denote this edge by $v_{3}v_{4}$).
Wlog (by symmetry), when Isolator plays again in this component he takes the edge $v_{1}v_{3}$, in which case Toucher can then immediately take the edge $v_{1}v_{4}$.
Whenever Isolator takes one of the two remaining edges in this component (note that both of these will be incident to $v_{2}$), Toucher can then immediately take the final edge
and will hence have touched all four vertices.
\end{proof}

\section{Discussion} \label{discussion}

Perhaps the most interesting unresolved issue concerns the asymptotic proportion of untouched vertices in $C_{n}$ and $P_{n}$.
We have shown 
in Theorem~\ref{cyc2} and Theorem~\ref{path1}
that this is somewhere between
$\frac{3}{16}$ and $\frac{1}{4}$,
but where exactly? Could it perhaps be $\frac{1}{5}$?
One intuitive reason for this is that Isolator needs two moves to isolate one vertex,
but Toucher can touch four vertices in this time, so we might expect that there should consequently be four times as many touched vertices as untouched. However, we have not managed to turn this reasoning into a formal argument.

Throughout this paper, whenever we have derived a bound, we have also tried to give tight examples that hold for infinitely many values of $n$.
However, in the case of our lower bound for $u(T)$ in Theorem~\ref{tree}, we only managed to provide one tight example, in Proposition~\ref{treeex2}.
Hence, it would be interesting to know whether there are other tight examples, or if in fact this lower bound can be improved for large $n$.
Also, what type of tree is most suitable for Toucher? Recall that we showed in Proposition~\ref{treeex1} that stars are the best choice for Isolator.

As we have seen in Remark~\ref{degseqremark} and Section~\ref{3reg},
we cannot hope to obtain exact results just by looking at the degree sequence of the graph. Hence, we are curious to know if any other properties or parameters of the graph can be utilised to give more precise bounds.

Finally, what is the largest possible proportion of untouched vertices for a $3$-regular graph? By Theorem~\ref{u>0example} and Corollary~\ref{3regcor}, we know that this is between $\frac{1}{24}$ and $\frac{1}{8}$.

\appendix
\section{Paths} \label{paths}

We now prove Theorem~\ref{path1} on paths
(recall that this result was used in the proofs of both 
Theorem~\ref{tree} and Proposition~\ref{treeex2}).
Clearly, the games on $P_n$ and $C_n$ are very closely related
(in fact, the game on $P_{n}$ is exactly equivalent to a game on $C_{n}$ in which Isolator has the first move),
and so the proofs here are similar to those for cycles.

\begin{proof}[Proof of upper bound of Theorem~\ref{path1}]
	For $1< n\leq 4$, knowing that Toucher is the first to play, the result follows easily. For $n>4$, we add a slight refinement to the analysis given in the proof of Theorem~\ref{gen1},
	again considering the strategy where Toucher always chooses the edge
	which maximises the sum of the Dangers of the two vertices incident to it.
	
	At the beginning of the game, the total Danger is $\sum_{v\in V(P_n)} 2^{-d(v)}=\frac{n+2}{4}$. Going through all possible cases, we see that Toucher will always decrease the total Danger by at least $\frac{6}{4}$ in her first two moves, while Isolator will only increase it by at most $\frac{5}{4}$ in his first two moves. Therefore, after these first two pairs of moves,
	the total Danger will have decreased by at least $\frac{1}{4}$. Thus, continuing as in the proof of Theorem~\ref{gen1},
	we hence obtain
	\begin{displaymath} u(P_n)\leq \sum_{v\in V(P_n)} 2^{-d(v)} - \frac{1}{4}=\frac{n+1}{4}. \qedhere
	\end{displaymath}
\end{proof}

\begin{Proposition}
	The graph $P_{3}$
	provides a tight example to the upper bound in Theorem~\ref{path1}.
	\qed
\end{Proposition}

For the lower bound, 
the key ingredient is Lemma~\ref{3of16lemma}, which enables Isolator to guarantee three untouched vertices from every sixteen edges.
As with $C_n$,
the main remaining issue is to deal with the leftover portion when the number of edges is not divisible by $16$.
This time, the argument is further complicated by the fact that Isolator will need to take advantage of the two leaves.

\begin{proof}[Proof of the lower bound of Theorem~\ref{path1}.]
	Let $k \in \{0,1,\ldots,15\}$
	denote the value of $(n-1) \bmod 16$.
	If $k \in \{0,1\}$,
	let $x=0$;
	if $k \in \{2,3,4,5,6\}$,
	let $x=1$;
	if $k \in \{7,8,9,10,11\}$,
	let $x=2$;
	and if $k \in \{12,13,14,15\}$,
	let $x=6$.
	Let $y = k-x \geq 0$.
	
	Before Toucher makes her first move,
	let Isolator partition the $n-1$ edges of $P_{n}$ into
	a `left-end' segment of $x$ consecutive edges,
	middle segments each of $16$ consecutive edges,
	and a `right-end' segment of $y$ consecutive edges.
	
	Let Isolator then use the strategy of always responding
	in the same segment in which Toucher played her previous move.
	By Lemma~\ref{3of16lemma},
	Isolator can thus guarantee that the number of untouched internal vertices
	in each $16$-edge segment
	will be at least three.
	
	Note that the statement of the theorem
	is equivalent to $u(P_{n}) \geq \frac{3}{16}(|E(P_{n})|-1)$.
	Hence,
	since $k$ is equal to the total number of edges in the two end segments,
	it now suffices to show that Isolator can guarantee
	isolating at least $\frac{3}{16}(k-1)$ vertices here.
	
	Throughout the remainder of the proof,
	note that we shall use the word `leaf' solely for the two leaves in $P_{n}$.
	Moreover,
	we shall not attempt to isolate
	the right-most vertex of the left-end segment
	or the left-most vertex
	of the right-most segment.
	
	If $k \leq 1$,
	then there is nothing to prove.
	
	If $k \in \{2,3,4,5,6\}$,
	then we need to show that Isolator can isolate at least one vertex.
	Recall $x=1$,
	so $y = k-x \geq 1$.
	Hence,
	as soon as Toucher takes an edge incident to one of the leaves,
	Isolator can simply take the edge incident to the other leaf,
	thus isolating it.
	
	If $k \in \{7,8,9,10,11\}$,
	then we need to show that Isolator can isolate at least two vertices.
	Recall $x=2$,
	so $y = k-x \geq 5$.
	As soon as Toucher takes an edge from one of the end segments
	(let us use $A$ to denote this segment),
	let Isolator take the edge incident to the leaf in the other end segment
	(let us use $B$ to denote this segment),
	thus isolating it.
	After Toucher's second move,
	we may assume that Toucher's two edges consist of
	the edge adjacent to Isolator's edge in Segment~$B$
	and the edge incident to the leaf in Segment~$A$
	(since otherwise Isolator could then take one of these,
	and we would be done).
	Hence,
	since $y \geq 5$,
	the right-end segment must certainly still contain three consecutive free edges,
	so we are done by Claim~\ref{cl1}.
	
	If $k \in \{12,13,14,15\}$,
	then we need to show that Isolator can isolate at least three vertices.
	Recall $x=6$,
	so $y \geq 6$ too.
	As soon as Toucher takes an edge from one of the end segments
	(let us again use $A$ to denote this segment),
	let Isolator take the edge incident to the leaf in the other end segment
	(let us again use $B$ to denote this segment),
	thus isolating it.
	
	Let us denote the first six edges in $A$,
	starting from the leaf,
	as $a_{1}, a_{2}, \ldots, a_{6}$,
	and let us similarly denote the first six edges in $B$,
	starting from the leaf,
	as $b_{1}, b_{2}, \ldots, b_{6}$.
	Hence,
	Isolator has claimed $b_{1}$.
	
	If Toucher's first two edges consist of $b_{2}$ and $a_{1}$,
	then $A$ still contains five consecutive free edges
	and $B$ still contains four consecutive free edges.
	By Claim~\ref{cl2},
	either Isolator can then isolate two internal vertice in $A$
	and is done,
	or he can isolate one internal vertex in $A$
	and then also one internal vertex in $B$
	(using Claim~\ref{cl1}),
	and is again done.
	
	If Toucher's first two edges are not $b_{2}$ and $a_{1}$,
	then Isolator may claim one of these with his second move,
	thus isolating a second vertex.
	If Isolator takes $b_{2}$,
	then we may assume that Toucher's first three edges include both $b_{3}$ and $a_{1}$
	(since otherwise Isolator could then also take one of these,
	and we would be done),
	so at least one of $A$ or $B$ will still contain three consecutive free edges,
	and so we may then just apply Claim~\ref{cl1}.
	Similarly,
	if Isolator takes $a_{1}$,
	then we may assume that Toucher's first three edges include both $a_{2}$ and $b_{2}$,
	and we can then use exactly the same argument.
\end{proof}

\begin{Proposition}
	The graph $P_{2}$
	provides a tight example to the lower bound in Theorem~\ref{path1}.
	\qed
\end{Proposition}

\end{document}